\documentclass[12pt,oneside]{article}

\usepackage{setspace}

\usepackage{amssymb}   

\usepackage{amsmath}

\usepackage{mathrsfs}

\usepackage{stmaryrd}

\usepackage{amsthm}
\usepackage{newlfont}
\usepackage{amscd}
\usepackage{mathtools}
\usepackage{bm}
\usepackage{tikz-cd}
\usepackage{graphicx}

\usepackage{hyperref}

\usepackage{enumitem}

\usepackage[all]{xy}

\usepackage{textcomp}

\usepackage{amsbsy}

\newtheorem{thm}{Theorem}[section]

\newtheorem{thm-defn}[thm]{Theorem/Definition}
\newtheorem{lem}[thm]{Lemma}
\newtheorem{prop}[thm]{Proposition}
\newtheorem{cor}[thm]{Corollary}

\theoremstyle{definition}
\newtheorem{defn}[thm]{Definition}

\theoremstyle{remark}

\numberwithin{equation}{section}

\begin{document}

\pagenumbering{arabic}

\title{Extending $p$-divisible groups and Barsotti-Tate deformation ring in the relative case}
\author{Yong Suk Moon}
\date{}

\maketitle

\begin{abstract}
Let $k$ be a perfect field of characteristic $p > 2$, and let $K$ be a finite totally ramified extension of $W(k)[\frac{1}{p}]$ of ramification degree $e$. We consider an unramified base ring $R_0$ over $W(k)$ satisfying certain conditions, and let $R = R_0\otimes_{W(k)}\mathcal{O}_K$. Examples of such $R$ include $R = \mathcal{O}_K[\![s_1, \ldots, s_d]\!]$ and $R = \mathcal{O}_K\langle t_1^{\pm 1}, \ldots, t_d^{\pm 1}\rangle$. We show that the generalization of Raynaud's theorem on extending $p$-divisible groups holds over the base ring $R$ when $e < p-1$, whereas it does not hold when $R = \mathcal{O}_K[\![s]\!]$ with $e \geq p$. As an application, we prove that if $R$ has Krull dimension $2$ and $e < p-1$, then the locus of Barsotti-Tate representations of $\mathrm{Gal}(\overline{R}[\frac{1}{p}]/R[\frac{1}{p}])$ cuts out a closed subscheme of the universal deformation scheme. If $R = \mathcal{O}_K[\![s]\!]$ with $e \geq p$, we prove that such a locus is not $p$-adically closed. 
\end{abstract}

\tableofcontents

\section{Introduction} \label{sec:1}

Let $k$ be a perfect field of characteristic $p > 2$, and $W(k)$ be its ring of Witt vectors. Let $K$ be a finite totally ramified extension of $W(k)[\frac{1}{p}]$ of ramification degree $e$, and let $\mathcal{O}_K$ be its ring of integers. We consider an unramified base ring $R_0$ over $W(k)$ satisfying certain conditions (cf. Section \ref{sec:2}), and let $R = R_0\otimes_{W(k)}\mathcal{O}_K$. Important examples of such $R$ include the formal power series ring $R = \mathcal{O}_K[\![s_1, \ldots, s_d]\!]$, and $R = \mathcal{O}_K\langle t_1^{\pm 1}, \ldots, t_d^{\pm 1}\rangle$ which is the $p$-adic completion of $\mathcal{O}_K[t_1^{\pm 1}, \ldots, t_d^{\pm 1}]$.  

When $R = \mathcal{O}_K$, Raynaud showed the following theorem on extending $p$-divisible groups.

\begin{thm} \label{thm:1.1} \emph{(\cite[Proposition 2.3.1]{raynaud-groupscheme})}
Let $G$ be a $p$-divisible group over $K$. Suppose that for each $n \geq 1$, $G[p^n]$ extends to a finite flat group scheme over $\mathcal{O}_K$. Then $G$ extends to a $p$-divisible group over $\mathcal{O}_K$, and such an extension is unique up to isomorphism.
\end{thm}

In this paper, we prove that the generalization of Raynaud's theorem holds over the relative base $R$ when the ramification is small $(e < p-1)$. On the other hand, using an example from \cite{vasiu-zink-purity} on purity of $p$-divisible groups, we show that such a statement does not hold when the ramification is large.

\begin{thm} \label{thm:1.2}
Assume $e < p-1$. Let $G$ be a $p$-divisible group over $R[\frac{1}{p}]$. Suppose that for each $n \geq 1$, $G[p^n]$ extends to a finite locally free group scheme over $R$. Then $G$ extends to a $p$-divisible group over $R$, and such an extension is unique up to isomorphism.	

If $e \geq p$ and $R = \mathcal{O}_K[\![s]\!]$, there exists a $p$-divisible group $G$ over $R[\frac{1}{p}]$ such that $G[p^n]$ extends to a finite locally free group scheme over $R$ for each $n$ but $G$ does not extend to a $p$-divisible group over $R$.
\end{thm}

As an application, we study the geometry of the locus of representations arising from $p$-divisible groups over $R$ when $R$ has Krull dimension $2$. Let $\mathcal{G}_R$ be the \'{e}tale fundamental group of $\mathrm{Spec} R[\frac{1}{p}]$. For a fixed absolutely irreducible $\mathbf{F}_p$-representation $V_0$ of $\mathcal{G}_R$, there exists a universal deformation ring which parametrizes the deformations of $V_0$ (\cite{smit}). We say that a finite continuous $\mathbf{Q}_p$-representation $V$ of $\mathcal{G}_R$ is \textit{Barsotti-Tate} if it arises from a $p$-divisible group over $R$, i.e., if there exists a $p$-divisible group $G_R$ over $R$ such that $V \cong T_p(G_R)\otimes_{\mathbf{Z}_p}\mathbf{Q}_p$ where $T_p(G_R)$ denotes the Tate module of $G_R$. For a torsion $\mathbf{Z}_p$-representation $T$ of $\mathcal{G}_R$, we say it is \textit{torsion Barsotti-Tate} if it is a quotient of a finite free $\mathbf{Z}_p$-representation $T_1$ such that $T_1[\frac{1}{p}]$ is Barsotti-Tate. By using Theorem \ref{thm:1.2}, we prove:

\begin{thm} \label{thm:1.3}
Suppose $R$ has Krull dimension $2$ and $e < p-1$. Then the locus of Barsotti-Tate representations of $\mathcal{G}_R$ cuts out a closed subscheme of the universal deformation scheme. 

If $R = \mathcal{O}_K[\![s]\!]$ and $e \geq p$, then the locus of Barsotti-Tate representations is not $p$-adically closed in the following sense: there exists a finite free $\mathbf{Z}_p$-representation $T$ of $\mathcal{G}_R$ such that $T/p^nT$ is torsion Barsotti-Tate for each integer $n \geq 1$ but $T[\frac{1}{p}]$ is not Barsotti-Tate.
\end{thm}

\noindent We give a more precise statement of Theorem \ref{thm:1.3} in Section \ref{sec:5}.

\section*{Acknowledgements}

I would like to express sincere gratitude to Mark Kisin for his guidance while working on this topic. This paper is partly based on author's Ph.D. thesis under his supervision. I also thank Brian Conrad and Tong Liu for helpful discussions.

\section{Relative Breuil-Kisin Classification} \label{sec:2}

We first explain the classification of $p$-divisible groups and finite locally free group schemes over $\mathrm{Spec} R$ via certain Kisin modules, which is proved in \cite{kisin-crystalline} when $R = \mathcal{O}_K$ and generalized in \cite{kim-groupscheme-relative} for the relative case. 

We will work over the relative base rings as considered in \cite{brinon-relative} with some additional mild assumptions. Denote by $W(k)\langle t_1^{\pm 1}, \ldots, t_d^{\pm 1}\rangle$ the $p$-adic completion of the polynomial ring $W(k)[t_1^{\pm 1}, \ldots, t_d^{\pm 1}]$. Let $R_0$ be a ring obtained from $W(k)\langle t_1^{\pm 1}, \ldots, t_d^{\pm 1}\rangle$ by iterations of the following operations:

\begin{itemize}
\item $p$-adic completion of an \'{e}tale extension;
\item $p$-adic completion of a localization;
\item completion with respect to an ideal containing $p$.	
\end{itemize}

\noindent We assume that either $W(k)\langle t_1^{\pm 1}, \ldots, t_d^{\pm 1}\rangle \rightarrow R_0$ has geometrically regular fibers or $R_0$ has Krull dimension less than $2$, and that $k \rightarrow R_0/pR_0$ is geometrically integral and $R_0$ is an integral domain. Furthermore, we suppose that $R_0$ is formally smooth formally finite type over some Cohen ring (cf. \cite[Section 2.2.2]{kim-groupscheme-relative}). In particular, $R_0$ is a regular ring.

$R_0/pR_0$ has a finite $p$-basis given by $\{t_1, \ldots, t_d\}$ in the sense of \cite[Definition 1.1.1]{deJong-dieudonnemodule}. Let $\widehat{\Omega}_{R_0} = \varprojlim_{n} \Omega_{(R_0/p^n)/W(k)}$ be the module of $p$-adically continuous K\"{a}hler differentials. We have $\widehat{\Omega}_{R_0} \cong \bigoplus_{i=1}^d R_0 \cdot d(\log{t_i})$ by \cite[Proposition 2.0.2]{brinon-relative}. The Witt vector Frobenius on $W(k)$ extends (not necessarily uniquely) to $R_0$. We fix such a Frobenius endomorphism $\varphi: R_0 \rightarrow R_0$, and let $R = R_0\otimes_{W(k)}\mathcal{O}_K$ be our base ring. Examples of such $R$ include $R = \mathcal{O}_K\langle t_1^{\pm 1}, \ldots, t_d^{\pm 1}\rangle$ and $R = \mathcal{O}_K[\![s_1, \ldots, s_d]\!]$ (for example, via $s_i = 1+t_i$). 

It will be useful later to consider the following natural maps between base rings. Let $R_{0, g}$ be the $p$-adic completion of $\displaystyle \varinjlim_{\varphi}(R_{0})_{(p)}$ with the induced Frobenius, and denote by $k_g$ the perfect closure $\displaystyle \varinjlim_{\varphi} \mathrm{Frac}(R_0/pR_0)$ of $\mathrm{Frac}(R_0/pR_0)$. By the universal property of $p$-adic Witt vectors, we have a unique continuous (with respect to the $p$-adic topology) morphism $h: W(k_g) \rightarrow R_{0, g}$ commuting with their projections to $k_g$. By unicity, $h$ is compatible with Frobenius endomorphisms. Since $h$ modulo $p$ is an isomorphism and $R_{0, g}$ is $p$-torsion free and $p$-adically complete and separated, $h$ is an isomorphism. We will make use of this isomorphism later when we apply results from classical $p$-adic Hodge theory over $p$-adic fields, since such results will hold for the base ring $R_{0, g}\otimes_{W(k)}\mathcal{O}_K$. Let $b_g: R_0 \rightarrow R_{0, g}$ be the natural morphism compatible with Frobenius. This induces $\mathcal{O}_K$-linearly the base change map $b_g: R \rightarrow R_{0, g}\otimes_{W(k)}\mathcal{O}_K$.

\begin{lem} \label{lem:2.1}
The map $b_g: R_0 \rightarrow R_{0, g}$ is injective. Furthermore, for each integer $n \geq 1$, the map $R_0/(p^n) \rightarrow R_{0, g}/(p^n)$ induced from $b_g$ is injective. 	
\end{lem}

\begin{proof}
Since $R_0/(p)$ is an integral domain, the map $R_0/(p) \rightarrow R_{0, g}/(p) = k_g$ is injective. Thus, $b_g: R_0 \rightarrow R_{0, g}$ is injective as $R_0$ is $p$-adically separated and $R_{0, g}$ is $p$-torsion free. It also follows that $R_0/(p^n) \rightarrow R_{0, g}/(p^n)$ is injective for each $n \geq 1$. 	
\end{proof}

Let $\mathfrak{S} = R_0[\![u]\!]$ equipped with the Frobenius extending that on $R_0$, given by $\varphi: u \mapsto u^p$. Denote by $E(u)$ the Eisenstein polynomial for the extension $K$ over $W(k)[\frac{1}{p}]$.

\begin{defn}
A \textit{quasi-Kisin module of height} $1$ is a pair $(\mathfrak{M}, \varphi_{\mathfrak{M}})$ where
\begin{itemize}
\item $\mathfrak{M}$ is a finitely generated projective $\mathfrak{S}$-module;
\item $\varphi_{\mathfrak{M}}: \mathfrak{M} \rightarrow \mathfrak{M}$ is a $\varphi$-semilinear map such that $\mathrm{coker}(1\otimes\varphi_{\mathfrak{M}}: \mathfrak{S}\otimes_{\varphi, \mathfrak{S}}\mathfrak{M} \rightarrow \mathfrak{M})$ is annihilated by $E(u)$.
\end{itemize} 	
\end{defn}

\noindent Note that for a quasi-Kisin module $\mathfrak{M}$ of height $1$, $1\otimes\varphi_{\mathfrak{M}}: \varphi^*\mathfrak{M} \coloneqq \mathfrak{S}\otimes_{\varphi, \mathfrak{S}}\mathfrak{M} \rightarrow \mathfrak{M}$ is injective since $\mathfrak{M}$ is finite projective over $\mathfrak{S}$ and $\mathrm{coker}(1\otimes\varphi_{\mathfrak{M}})$ is killed by $E(u)$. Let $\mathrm{Mod}_{\mathfrak{S}}(\varphi)$ denote the category of quasi-Kisin modules of height $1$ whose morphisms are $\mathfrak{S}$-module maps compatible with Frobenius. 

Consider the composite $\mathfrak{S} \twoheadrightarrow \mathfrak{S}/u\mathfrak{S} = R_0 \stackrel{\varphi}{\rightarrow} R_0$. Let $\mathrm{Mod}_{\mathfrak{S}}(\varphi, \nabla)$ denote the category whose objects are tuples $(\mathfrak{M}, \varphi_{\mathfrak{M}}, \nabla_{\mathcal{M}})$ where $(\mathfrak{M}, \varphi_{\mathfrak{M}})$ is a quasi-Kisin module of height $1$, $\mathcal{M} \coloneqq \mathfrak{M}\otimes_{\mathfrak{S}, \varphi}R_0$, and $\nabla_{\mathcal{M}}: \mathcal{M} \rightarrow \mathcal{M}\otimes_{R_0}\widehat{\Omega}_{R_0}$ is a topologically quasi-nilpotent integrable connection commuting with $\varphi_{\mathcal{M}} \coloneqq \varphi_{\mathfrak{M}}\otimes\varphi_{R_0}$. (Here, $\nabla_{\mathcal{M}}$ being topologically quasi-nilpotent means that the induced connection on $\mathcal{M}/p\mathcal{M}$ is nilpotent). The morphisms in $\mathrm{Mod}_{\mathfrak{S}}(\varphi, \nabla)$ are $\mathfrak{S}$-module maps compatible with Frobenius and connection. The objects in $\mathrm{Mod}_{\mathfrak{S}}(\varphi, \nabla)$ are called \textit{Kisin modules of height} $1$. The following theorem is proved in \cite{kim:15}.

\begin{thm} \label{thm:2.3} \emph{(cf. \cite[Corollary 6.7 and Remark 6.9]{kim-groupscheme-relative})}
There exists an exact anti-equivalence of categories 
\[
\mathfrak{M}^*: \{p\mbox{-divisible groups over } R\} \rightarrow \mathrm{Mod}_{\mathfrak{S}}(\varphi, \nabla).
\]	
Let $R_0'$ be another unramifed ring satisfying the conditions as above equipped with a Frobenius, and let $b: R_0 \rightarrow R_0'$ be a $\varphi$-equivariant map. Then the formation of $\mathfrak{M}^*$ commutes with the base change $R \rightarrow R' \coloneqq R_0'\otimes_{W(k)}\mathcal{O}_K$ induced $\mathcal{O}_K$-linearly from $b$.
\end{thm}

The classification of $p$-power order finite locally free group schemes over $R$ can be obtained by considering torsion Kisin modules. 

\begin{defn}
A \textit{torsion quasi-Kisin module of height} $1$ is a pair $(\mathfrak{M}, \varphi_{\mathfrak{M}})$ where
\begin{itemize}
\item $\mathfrak{M}$ is a finitely presented $\mathfrak{S}$-module killed by a power of $p$, and of $\mathfrak{S}$-projective dimension $1$;
\item $\varphi_{\mathfrak{M}}: \mathfrak{M} \rightarrow \mathfrak{M}$ is a $\varphi$-semilinear endomorphism such that $\mathrm{coker}(1\otimes\varphi_{\mathfrak{M}}: \varphi^*\mathfrak{M} \rightarrow \mathfrak{M})$ is killed by $E(u)$. 
\end{itemize}	
\end{defn}

Let $\mathrm{Mod}_{\mathfrak{S}}^{\mathrm{tor}}(\varphi)$ denote the category of torsion quasi-Kisin modules of height $1$ whose morphisms are $\mathfrak{S}$-linear maps compatible with $\varphi$. Let $\mathrm{Mod}_{\mathfrak{S}}^{\mathrm{tor}}(\varphi, \nabla)$ denote the category whose objects are tuples $(\mathfrak{M}, \varphi_{\mathfrak{M}}, \nabla_{\mathcal{M}})$ where $(\mathfrak{M}, \varphi_{\mathfrak{M}})$ is a torsion quasi-Kisin module of height $1$, $\mathcal{M} \coloneqq \mathfrak{M}\otimes_{\mathfrak{S}, \varphi}R_0$, and $\nabla_{\mathcal{M}}: \mathcal{M} \rightarrow \mathcal{M}\otimes_{R_0}\widehat{\Omega}_{R_0}$ is a topologically quasi-nilpotent integrable connection commuting with $\varphi_{\mathcal{M}} \coloneqq \varphi_{\mathfrak{M}}\otimes\varphi_{R_0}$. The morphisms in $\mathrm{Mod}_{\mathfrak{S}}^{\mathrm{tor}}(\varphi, \nabla)$ are $\mathfrak{S}$-module maps compatible with $\varphi$ and $\nabla$. The objects are called \textit{torsion Kisin modules of height} $1$.

\begin{lem} \label{lem:2.5}
Let $\mathfrak{M}$ be a torsion quasi-Kisin module of height $1$. Then $1\otimes\varphi_{\mathfrak{M}}: \varphi^*\mathfrak{M} \rightarrow \mathfrak{M}$ is injective.
\end{lem}

\begin{proof}
Let $\mathfrak{S}_g \coloneqq R_{0, g}[\![u]\!]$ equipped with the Frobenius given by $\varphi(u) = u^p$. By the local criterion for flatness, $b_g: R_0 \rightarrow R_{0, g}$ is flat since $R_0/(p) \rightarrow R_{0, g}/(p) = k_g$ is flat and $R_{0, g}$ is $p$-torsion free, and the map $\mathfrak{S} \rightarrow \mathfrak{S}_g$ is flat. Note that $\mathfrak{M}_g \coloneqq \mathfrak{M}\otimes_{\mathfrak{S}}\mathfrak{S}_g$ equipped with $\varphi_{\mathfrak{M}_g} \coloneqq \varphi_{\mathfrak{M}}\otimes\varphi_{\mathfrak{S}_g}$ is a torsion Kisin module of height $1$ over $\mathfrak{S}_g$. 

We first claim that the natural map $b: \mathfrak{M} \rightarrow \mathfrak{M}_g$ is injective. Since $\mathfrak{M}$ has projective dimension $\leq 1$, there exists a short exact sequence $0 \rightarrow \mathfrak{M}_1 \rightarrow \mathfrak{M}_2 \rightarrow \mathfrak{M} \rightarrow 0$ where $\mathfrak{M}_1$ and $\mathfrak{M}_2$ are finite projective $\mathfrak{S}$-modules. $\mathfrak{M}_1$ and $\mathfrak{M}_2$ have the same rank since $\mathfrak{M}$ is killed by a power of $p$. We have a commutative diagram
\[
\begin{CD}
	0 @>>> \mathfrak{M}_1 @>>> \mathfrak{M}_2 @>>> \mathfrak{M} @>>> 0\\
	&&		@VVV 				@VVV				@VVbV\\
	0 @>>> \mathfrak{M}_1\otimes_{\mathfrak{S}}\mathfrak{S}_g @>>> \mathfrak{M}_2\otimes_{\mathfrak{S}}\mathfrak{S}_g @>>> \mathfrak{M}_g @>>> 0
\end{CD}
\]
whose rows are exact. Since $\mathfrak{M}_1$ and $\mathfrak{M}_2$ are projective over $\mathfrak{S}$, the left and middle vertical maps are injective. Furthermore, for $i = 1, 2$, we have $\mathrm{coker}(\mathfrak{M}_i \rightarrow \mathfrak{M}_i\otimes_{\mathfrak{S}}\mathfrak{S}_g) \cong \mathfrak{M}_i\otimes_{\mathfrak{S}}(\mathfrak{S}_g/\mathfrak{S})$ as $\mathfrak{S}$-modules. On the other hand, all elements in the kernel of the induced map $\mathfrak{M}_1\otimes_{\mathfrak{S}}(\mathfrak{S}_g/\mathfrak{S}) \rightarrow \mathfrak{M}_2\otimes_{\mathfrak{S}}(\mathfrak{S}_g/\mathfrak{S})$ are killed by some power of $p$ since $\mathfrak{M}_1[\frac{1}{p}] \cong \mathfrak{M}_2[\frac{1}{p}]$. And $\mathfrak{S}_g/\mathfrak{S}$ is $p$-torsion free since $R_0/(p) \rightarrow R_{0, g}/(p) = k_g$ is injective, so $\mathfrak{M}_1\otimes_{\mathfrak{S}}(\mathfrak{S}_g/\mathfrak{S})$ is $p$-torsion free as $\mathfrak{M}_1$ is projective over $\mathfrak{S}$. Hence, the map $\mathfrak{M}_1\otimes_{\mathfrak{S}}(\mathfrak{S}_g/\mathfrak{S}) \rightarrow \mathfrak{M}_2\otimes_{\mathfrak{S}}(\mathfrak{S}_g/\mathfrak{S})$ is injective. By the snake Lemma, we deduce that $b: \mathfrak{M} \rightarrow \mathfrak{M}_g$ is injective.    

Now, consider the following commutative diagram:
\[
\begin{CD}
	\mathfrak{S}\otimes_{\varphi, \mathfrak{S}}\mathfrak{M} @>1\otimes\varphi_{\mathfrak{M}}>> \mathfrak{M}\\
	@VVV                       @VVbV\\
	\mathfrak{S}_g\otimes_{\varphi, \mathfrak{S}_g}\mathfrak{M}_g @>1\otimes\varphi_{\mathfrak{M}_g}>> \mathfrak{M}_g   
\end{CD}
\]
Since $\varphi: \mathfrak{S} \rightarrow \mathfrak{S}$ is flat by \cite[Lemma 7.1.8]{brinon-relative}, $\mathfrak{S}\otimes_{\varphi, \mathfrak{S}}\mathfrak{M}$ has projective dimension $1$ as a $\mathfrak{S}$-module and is killed by a power of $p$. By the same argument as above, the natural map  $\mathfrak{S}\otimes_{\varphi, \mathfrak{S}}\mathfrak{M} \rightarrow \mathfrak{S}_g\otimes_{\mathfrak{S}}(\mathfrak{S}\otimes_{\varphi, \mathfrak{S}}\mathfrak{M}) \cong \mathfrak{S}_g\otimes_{\varphi, \mathfrak{S}_g}\mathfrak{M}_g$ is injective, which is the left vertical map. The bottom map is injective by \cite[Proposition 2.3.2]{liu-fontaineconjecture} since $R_{0, g} \cong W(k_g)$. Thus, the top map is injective.
\end{proof}

Denote by $(\mathrm{Mod~FI})_{\mathfrak{S}}(\varphi, \nabla)$ the full subcategory of $\mathrm{Mod}_{\mathfrak{S}}^{\mathrm{tor}}(\varphi, \nabla)$ consisting of objects $\mathfrak{M}$ such that $\mathfrak{M} \cong \bigoplus_i \mathfrak{M}_i$ as $\mathfrak{S}$-modules where $\mathfrak{M}_i$'s are projective over $\mathfrak{S}/(p^{n_i})$ for some positive integers $n_i$. The following theorem is shown in \cite{kim-groupscheme-relative}.

\begin{thm} \label{thm:2.6} \emph{(cf. \cite[Proposition 9.5 and Theorem 9.8]{kim-groupscheme-relative})}
There exists an exact fully faithful functor $\mathfrak{M}^*$ from the category of $p$-power order finite locally free group schemes over $R$ to $\mathrm{Mod}_{\mathfrak{S}}^{\mathrm{tor}}(\varphi, \nabla)$ with the following properties:
\begin{itemize}
\item Let $H$ be a $p$-power order finite locally free group scheme over $R$. If $H = \mathrm{ker}(h: G^0 \rightarrow G^1)$ for an isogeny $h$ of $p$-divisible groups over $R$, then there exists a natural isomorphism $\mathfrak{M}^*(H) \cong \mathrm{coker}(\mathfrak{M}^*(h))$ of torsion Kisin modules of height $1$;
\item Let $R_0'$ be another unramified ring satisfying the conditions as above equipped with a Frobenius, and let $b: R_0 \rightarrow R_0'$ be a $\varphi$-equivariant map. Then the formation of $\mathfrak{M}^*$ commutes with the base change $R \rightarrow R' \coloneqq R_0'\otimes_{W(k)}\mathcal{O}_K$ induced $\mathcal{O}_K$-linearly from $b$. 	
\end{itemize}
Moreover, the functor $\mathfrak{M}^*$ induces an anti-equivalence from the category of $p$-power order finite locally free group schemes $H$ over $R$ such that $H[p^n]$ is locally free over $R$ for all $n \geq 1$ to $(\mathrm{Mod~FI})_{\mathfrak{S}}(\varphi, \nabla)$.
\end{thm}

We end this section by recalling some necessary results on connections explained in \cite[Section 10.2]{kim-groupscheme-relative}, which is based on \cite{vasiu}. Let $(\mathfrak{M}, \varphi_{\mathfrak{M}})$ be a quasi-Kisin module of height $1$, and let $\mathcal{M} = \mathfrak{M}\otimes_{\mathfrak{S}, \varphi}R_0$ equipped with the induced Frobenius $\varphi_{\mathfrak{M}}\otimes\varphi_{R_0}$. From \cite[Eq. (6.1), (6.2) and Remark 3.13]{kim-groupscheme-relative}, we have the $R_0$-submodule $\mathrm{Fil}^1 \mathcal{M} \subset \mathcal{M}$ associated with $\mathfrak{M}$ such that $p\mathcal{M} \subset \mathrm{Fil}^1 \mathcal{M}$, $\mathcal{M}/\mathrm{Fil}^1 \mathcal{M}$ is projective over $R_0/(p)$, and $(1\otimes\varphi)(\varphi^*\mathrm{Fil}^1 \mathcal{M}) = p\mathcal{M}$ as $R_0$-modules (cf. \cite[Definition 3.4 and 3.6]{kim-groupscheme-relative} for the frame $(R_0, pR_0, R_0/(p), \varphi_{R_0}, \frac{\varphi_{R_0}}{p})$). Fix an $R_0$-direct factor $\mathcal{M}^1 \subset \mathcal{M}$ which lifts $\mathrm{Fil}^1 \mathcal{M}/p\mathcal{M} \subset \mathcal{M}/p\mathcal{M}$, and let $\tilde{M} \coloneqq (\mathcal{M}+\frac{1}{p}\mathcal{M}^1)\otimes_{R_0, \varphi}R_0 \subset \mathcal{M}\otimes_{R_0, \varphi}R_0[\frac{1}{p}]$. For each integer $n \geq 1$, suppose $\nabla_n: R_0/(p^n)\otimes_{R_0}\mathcal{M} \rightarrow (R_0/(p^n)\otimes_{R_0}\mathcal{M})\otimes_{R_0}\widehat{\Omega}_{R_0}$ is a connection such that the following diagram is commutative: 
\begin{equation} \label{eq:2.1}
\begin{CD}
R_0/(p^n)\otimes_{R_0}\tilde{\mathcal{M}} @>\varphi^*(\nabla_n)>> R_0/(p^n)\otimes_{R_0}\tilde{\mathcal{M}}\otimes_{R_0}\widehat{\Omega}_{R_0}\\
@V1\otimes\varphi VV   @VV(1\otimes\varphi)\otimes \mathrm{id}_{\widehat{\Omega}_{R_0}}V\\
R_0/(p^n)\otimes_{R_0}\mathcal{M} @>\nabla_n>> R_0/(p^n)\otimes_{R_0}\mathcal{M}\otimes_{R_0}\widehat{\Omega}_{R_0}	
\end{CD}
\end{equation}
Here, $\varphi^*(\nabla_n)$ is given by choosing an arbitrary lift of $\nabla_n$ on $R_0/(p^{n+1})\otimes_{R_0}\mathcal{M}$, and $\varphi^*(\nabla_n)$ does not depend on the choice of such a  lift (cf. \cite[Section 3.1.1 Equation (9)]{vasiu}). Identify $\widehat{\Omega}_{R_0} =\bigoplus_{i=1}^d R_0 \cdot d(\log{t_i})$. By passing to a finite Zariski covering of $\mathrm{Spf}(R_0, p)$, we may assume that $\mathcal{M}^1$ and $\mathcal{M}/\mathcal{M}^1$ are free over $R_0$. Fix such a choice of the covering, and fix a $R_0$-basis of $\mathcal{M}$ adapted to the direct factor $\mathcal{M}^1$. By \cite[Section 3.2 Basic Theorem]{vasiu} and its proof, the set of connections $\nabla_1$ on $R_0/(p)\otimes_{R_0} \mathcal{M}$ satisfying the commutative diagram (\ref{eq:2.1}) for $n = 1$ corresponds to the solutions over $R_0/(p)$ of a certain Artin-Schreier system of equations over $R_0/(p)$. In particular, it follows directly that we have finitely many such $\nabla_1$ (cf. \cite[Theorem 2.4.1 (b)]{vasiu}). Furthermore, given a connection $\nabla_n$ on $R_0/(p^n)\otimes_{R_0}\mathcal{M}$, the set of connections $\nabla_{n+1}$ on $R_0/(p^{n+1})\otimes_{R_0}\mathcal{M}$ which lift $\nabla_n$ and satisfy the commutative diagram (\ref{eq:2.1}) for $n+1$ corresponds the solutions over $R_0/(p)$ of a certain Artin-Schreier system of equations over $R_0/(p)$ by \textit{loc. cit.}, and we have finitely many such $\nabla_{n+1}$.

\section{\'{E}tale $\varphi$-modules and Galois Representations} \label{sec:3}

We recall the results in \cite[Section 7]{kim-groupscheme-relative} on associating Galois representations with \'{e}tale $\varphi$-modules in the relative setting. The underlying geometry is based on perfectoid spaces (cf. \cite{scholze-perfectoid}). We will use the results to translate our question on $p$-divisible groups into a question on Kisin modules and \'etale $\varphi$-modules. 

Let $\overline{R}$ denote the union of finite $R$-subalgebras $R'$ of a fixed separable closure of $\mathrm{Frac}(R)$ such that $R'[\frac{1}{p}]$ is \'{e}tale over $R[\frac{1}{p}]$. Then $\mathrm{Spec}\overline{R}[\frac{1}{p}]$ is a pro-universal covering of $\mathrm{Spec}R[\frac{1}{p}]$, and $\overline{R}$ is the integral closure of $R$ in $\overline{R}[\frac{1}{p}]$. Let $\mathcal{G}_R \coloneqq \mathrm{Gal}(\overline{R}[\frac{1}{p}]/R[\frac{1}{p}]) = \pi_1^{\text{\'{e}t}}(\mathrm{Spec}R[\frac{1}{p}], \eta)$ with a choice of a geometric point $\eta$. Choose a uniformizer $\varpi \in \mathcal{O}_K$. For integers $n \geq 0$, we choose compatibly $\varpi_n \in \overline{R}$ such that $\varpi_0 = \varpi$ and $\varpi_{n+1}^p = \varpi_n$, and let $L$ be the $p$-adic completion of $\bigcup_{n \geq 0} K(\varpi_n)$. Then $L$ is a perfectoid field and $(\widehat{\overline{R}}[\frac{1}{p}], \widehat{\overline{R}})$ is a perfectoid affinoid $L$-algebra, where $\widehat{\overline{R}}$ denotes the $p$-adic completion of $\overline{R}$.

Let $L^\flat$ denote the tilt of $L$ as defined in \cite{scholze-perfectoid}, and let $\underline{\varpi} \coloneqq (\varpi_n) \in L^\flat$. Let $(\overline{R}^\flat[\frac{1}{\underline{\varpi}}], \overline{R}^\flat)$ be the tilt of $(\widehat{\overline{R}}[\frac{1}{p}], \widehat{\overline{R}})$. Let $E_{R_\infty}^+ = \mathfrak{S}/p\mathfrak{S}$, and let $\tilde{E}_{R_\infty}^+$ be the $u$-adic completion of $\varinjlim_{\varphi}E_{R_\infty}^+$. Let $E_{R_\infty} = E_{R_\infty}^+[\frac{1}{u}]$ and $\tilde{E}_{R_\infty} = \tilde{E}_{R_\infty}^+[\frac{1}{u}]$. By \cite[Proposition 5.9]{scholze-perfectoid}, $(\tilde{E}_{R_\infty}, \tilde{E}_{R_\infty}^+)$ is a perfectoid affinoid $L^\flat$-algebra, and we have the natural injection $(\tilde{E}_{R_\infty}, \tilde{E}_{R_\infty}^+) \hookrightarrow (\overline{R}^\flat[\frac{1}{\underline{\varpi}}], \overline{R}^\flat)$ given by $u \mapsto \underline{\varpi}$. Let $(\tilde{R}_\infty[\frac{1}{p}], \tilde{R}_\infty)$ be a perfectoid affinoid $L$-algebra whose tilt is $(\tilde{E}_{R_\infty}, \tilde{E}_{R_\infty}^+)$, and let $\mathcal{G}_{\tilde{R}_\infty} = \pi_1^{\text{\'{e}t}}(\mathrm{Spec}\tilde{R}_\infty[\frac{1}{p}], \eta)$. Then we have a continuous map of Galois groups $\mathcal{G}_{\tilde{R}_\infty} \rightarrow \mathcal{G}_R$, which is a closed embedding by \cite[Proposition 5.4.54]{gabber-almost}. By the almost purity theorem in \cite{scholze-perfectoid}, $\overline{R}^\flat[\frac{1}{\underline{\varpi}}]$ can be canonically identified with the $\underline{\varpi}$-adic completion of the affine ring of a pro-universal covering of $\mathrm{Spec}\tilde{E}_{R_\infty}$, and letting $\mathcal{G}_{\tilde{E}_{R_\infty}}$ be the Galois group corresponding to the pro-universal covering, there exists a canonical isomorphism $\mathcal{G}_{\tilde{E}_{R_\infty}} \cong \mathcal{G}_{\tilde{R}_\infty}$.

Now, let $\mathcal{O}_{\mathcal{E}}$ be the $p$-adic completion of $\mathfrak{S}[\frac{1}{u}]$. Note that $\varphi$ on $\mathfrak{S}$ extends naturally to $\mathcal{O}_{\mathcal{E}}$.

\begin{defn}
An \textit{\'{e}tale} $(\varphi, \mathcal{O}_\mathcal{E})$-\textit{module} is a pair $(M, \varphi_M)$ where $M$ is a finitely generated $\mathcal{O}_\mathcal{E}$-module and $\varphi_M: M \rightarrow M$ is a $\varphi$-semilinear endomorphism such that $1\otimes\varphi_M: \varphi^*M \rightarrow M$ is an isomorphism. We say that an \'{e}tale $(\varphi, \mathcal{O}_{\mathcal{E}})$-module is \textit{projective} (resp. \textit{torsion}) if the underlying $\mathcal{O}_{\mathcal{E}}$-module $M$ is projective (resp. $p$-power torsion).
\end{defn}

\noindent Let $\mathrm{Mod}_{\mathcal{O}_{\mathcal{E}}}$ denote the category of \'{e}tale $(\varphi, \mathcal{O}_{\mathcal{E}})$-modules whose morphisms are $\mathcal{O}_\mathcal{E}$-linear maps compatible with Frobenius. Let  $\mathrm{Mod}_{\mathcal{O}_{\mathcal{E}}}^{\mathrm{pr}}$ and  $\mathrm{Mod}_{\mathcal{O}_{\mathcal{E}}}^{\mathrm{tor}}$ respectively denote the full subcategories of projective and torsion objects. 

Note that we have a natural notion of a subquotient, direct sum, and tensor product for \'{e}tale $(\varphi, \mathcal{O}_\mathcal{E})$-modules, and duality is defined for projective and torsion objects. If $(\mathfrak{M}, \varphi_{\mathfrak{M}})$ is a quasi-Kisin module (resp. torsion quasi-Kisin module) of height $1$, then $(\mathfrak{M}\otimes_{\mathfrak{S}}\mathcal{O}_{\mathcal{E}}, \varphi_{\mathfrak{M}}\otimes\varphi_{\mathcal{O}_{\mathcal{E}}})$ is a projective (resp. torsion) \'{e}tale $(\varphi, \mathcal{O}_{\mathcal{E}})$-module since $1\otimes\varphi_{\mathfrak{M}}$ is injective (by Lemma \ref{lem:2.5} for torsion quasi-Kisin modules) and its cokernel is killed by $E(u)$ which is a unit in $\mathcal{O}_\mathcal{E}$. If we denote by $\mathcal{O}_{\mathcal{E}, g}$ the corresponding ring for $R_{0, g}$, then for any \'{e}tale $(\varphi, \mathcal{O}_\mathcal{E})$-module $M$, $M\otimes_{\mathcal{O}_\mathcal{E}, b_g} \mathcal{O}_{\mathcal{E}, g}$ with the induced Frobenius is an \'{e}tale $(\varphi, \mathcal{O}_{\mathcal{E}, g})$-module. If $M$ is a torsion object, we define its \textit{length} to be the length of $\mathcal{O}_{\mathcal{E}, g}$-module $M\otimes_{\mathcal{O}_\mathcal{E}, b_g} \mathcal{O}_{\mathcal{E}, g}$.

We consider $W(\overline{R}^\flat[\frac{1}{\underline{\varpi}}])$ as an $\mathcal{O}_\mathcal{E}$-algebra via mapping $u$ to the Teichm\"uller lift $[\underline{\varpi}]$ of $\underline{\varpi}$, and let $\mathcal{O}_\mathcal{E}^{\mathrm{ur}}$ be the integral closure of $\mathcal{O}_\mathcal{E}$ in $W(\overline{R}^\flat[\frac{1}{\underline{\varpi}}])$. Let $\widehat{\mathcal{O}}_{\mathcal{E}}^{\mathrm{ur}}$ be its $p$-adic completion. Since $\mathcal{O}_\mathcal{E}$ is normal, we have $\mathrm{Aut}_{\mathcal{O}_\mathcal{E}}(\mathcal{O}_\mathcal{E}^{\mathrm{ur}}) \cong \mathcal{G}_{E_{R_\infty}} \coloneqq \pi_1^{\text{\'et}}(\mathrm{Spec}E_{R_\infty})$, and by \cite[Proposition 5.4.54]{gabber-almost} and the almost purity theorem, we have $\mathcal{G}_{E_{R_\infty}} \cong \mathcal{G}_{\tilde{E}_{R_\infty}} \cong \mathcal{G}_{\tilde{R}_\infty}$. This induces $\mathcal{G}_{\tilde{R}_\infty}$-action on $\widehat{\mathcal{O}}_{\mathcal{E}}^{\mathrm{ur}}$. The following is shown in \cite{kim-groupscheme-relative}.

\begin{lem} \label{lem:3.2} \emph{(cf. \cite[Lemma 7.5 and 7.6]{kim-groupscheme-relative})}
We have $(\widehat{\mathcal{O}}_{\mathcal{E}}^{\mathrm{ur}})^{\mathcal{G}_{\tilde{R}_\infty}} = \mathcal{O}_\mathcal{E}$ and the same holds modulo $p^n$. Furthermore, there exists a unique $\mathcal{G}_{\tilde{R}_\infty}$-equivariant 	ring endomorphism $\varphi$ on $\widehat{\mathcal{O}}_{\mathcal{E}}^{\mathrm{ur}}$ lifting the $p$-th power map on $\widehat{\mathcal{O}}_{\mathcal{E}}^{\mathrm{ur}}/(p)$ and extending $\varphi$ on $\mathcal{O}_\mathcal{E}$. The inclusion $\widehat{\mathcal{O}}_{\mathcal{E}}^{\mathrm{ur}} \hookrightarrow W(\overline{R}^\flat[\frac{1}{\underline{\varpi}}])$ is $\varphi$-equivariant where the latter ring is given the Witt vector Frobenius.
\end{lem}
 
Let $\mathrm{Rep}_{\mathbf{Z}_p}(\mathcal{G}_{\tilde{R}_\infty})$ be the category of finite continuous $\mathbf{Z}_p$-representations of $\mathcal{G}_{\tilde{R}_\infty}$, and let $\mathrm{Rep}_{\mathbf{Z}_p}^{\mathrm{free}}(\mathcal{G}_{\tilde{R}_\infty})$ and $\mathrm{Rep}_{\mathbf{Z}_p}^{\mathrm{tor}}(\mathcal{G}_{\tilde{R}_\infty})$ respectively denote the full subcategories of free and torsion objects. For $M \in \mathrm{Mod}_{\mathcal{O}_\mathcal{E}}$ and $T \in \mathrm{Rep}_{\mathbf{Z}_p}(\mathcal{G}_{\tilde{R}_\infty})$, we define $T(M) \coloneqq (M\otimes_{\mathcal{O}_\mathcal{E}}\widehat{\mathcal{O}}_{\mathcal{E}}^{\mathrm{ur}})^{\varphi = 1}$ and $M(T) \coloneqq (T\otimes_{\mathbf{Z}_p}\widehat{\mathcal{O}}_{\mathcal{E}}^{\mathrm{ur}})^{\mathcal{G}_{\tilde{R}_\infty}}$. Then we have the following proposition from \cite{kim-groupscheme-relative}.

\begin{prop} \label{prop:3.3} \emph{(\cite[Proposition 7.7]{kim-groupscheme-relative})}
The constructions $T(\cdot)$ and $M(\cdot)$ give exact quasi-inverse equivalences of $\otimes$-categories between $\mathrm{Mod}_{\mathcal{O}_\mathcal{E}}$ and $\mathrm{Rep}_{\mathbf{Z}_p}(\mathcal{G}_{\tilde{R}_\infty})$. Moreover, $T(\cdot)$ and $M(\cdot)$ restrict to rank-preserving equivalences of categories between $\mathrm{Mod}_{\mathcal{O}_\mathcal{E}}^{\mathrm{pr}}$ and $\mathrm{Rep}_{\mathbf{Z}_p}^{\mathrm{free}}(\mathcal{G}_{\tilde{R}_\infty})$, and length-preserving equivalences between $\mathrm{Mod}_{\mathcal{O}_\mathcal{E}}^{\mathrm{tor}}$ and $\mathrm{Rep}_{\mathbf{Z}_p}^{\mathrm{tor}}(\mathcal{G}_{\tilde{R}_\infty})$. In both cases, $T(\cdot)$ and $M(\cdot)$ commute with taking duals.  
\end{prop}

For $M$ in $\mathrm{Mod}_{\mathcal{O}_\mathcal{E}}^{\mathrm{pr}}$ (resp. in $\mathrm{Mod}_{\mathcal{O}_\mathcal{E}}^{\mathrm{tor}}$), we define the contravariant functor $T^\vee(\cdot)$ to $\mathrm{Rep}_{\mathbf{Z}_p}(\mathcal{G}_{\tilde{R}_\infty})$ by $T^\vee(M) \coloneqq \mathrm{Hom}_{\mathcal{O}_\mathcal{E}, \varphi}(M, \widehat{\mathcal{O}}_{\mathcal{E}}^{\mathrm{ur}})$ (resp. $\mathrm{Hom}_{\mathcal{O}_\mathcal{E}, \varphi}(M, \widehat{\mathcal{O}}_{\mathcal{E}}^{\mathrm{ur}}\otimes_{\mathbf{Z}_p}\mathbf{Q}_p/\mathbf{Z}_p)$). Note that if we have a short exact sequence of \'etale $(\varphi, \mathcal{O}_\mathcal{E})$-modules $0 \rightarrow M_1 \rightarrow M_2 \rightarrow M \rightarrow 0$ where $M_1, M_2$ are projective over $\mathcal{O}_\mathcal{E}$ and $M$ is $p$-power torsion, then it induces a short exact sequence 
\[
0 \rightarrow T^\vee(M_2) \rightarrow T^\vee(M_1) \rightarrow T^\vee(M) \rightarrow 0
\] 
in $\mathrm{Rep}_{\mathbf{Z}_p}(\mathcal{G}_{\tilde{R}_\infty})$.

Now, if $G_R$ is a $p$-divisible group over $R$, we write $T_p(G_R) \coloneqq \mathrm{Hom}_{\overline{R}}(\mathbf{Q}_p/\mathbf{Z}_p, G_R\times_R \overline{R})$ to be the associated Tate module, which is a finite free $\mathbf{Z}_p$-representation of $\mathcal{G}_R$. By \cite[Corollary 8.2]{kim-groupscheme-relative}, we have a natural $\mathcal{G}_{\tilde{R}_\infty}$-equivariant isomorphism $T^\vee(\mathfrak{M}^*(G_R)\otimes_{\mathfrak{S}}\mathcal{O}_\mathcal{E}) \cong T_p(G_R)$. If $H$ is a $p$-power order finite locally free group scheme over $R$, then $H(\overline{R})$ is a finite torsion $\mathbf{Z}_p$-representation of $\mathcal{G}_R$. By \cite[Proposiiton 9.10]{kim-groupscheme-relative}, there exists a natural $\mathcal{G}_{\tilde{R}_\infty}$-equivariant isomorphism $T^\vee(\mathfrak{M}^*(H)\otimes_{\mathfrak{S}}\mathcal{O}_\mathcal{E}) \cong H(\overline{R})$, and if $H = \mathrm{ker}(h: G^0 \rightarrow G^1)$ for some isogeny $h$ of $p$-divisible groups over $R$, then the isomorphism $T^\vee(\mathfrak{M}^*(H)\otimes_{\mathfrak{S}}\mathcal{O}_\mathcal{E}) \cong H(\overline{R})$ is compatible with the isomorphisms $T^\vee(\mathfrak{M}^*(G^i)\otimes_{\mathfrak{S}}\mathcal{O}_\mathcal{E}) \cong T_p(G^i), ~i = 0, 1$.

Note that any $p$-divisible group over $R[\frac{1}{p}]$ is \'{e}tale, so the category of $p$-divisible groups over $R[\frac{1}{p}]$ is equivalent to the category of finite free $\mathbf{Z}_p$-representations of $\mathcal{G}_R$. If we are given a $p$-divisible group $G$ over $R[\frac{1}{p}]$, then the corresponding Galois representation is given by $T_p(G) = \mathrm{Hom}_{\overline{R}[\frac{1}{p}]}(\mathbf{Q}_p/\mathbf{Z}_p, G\times_{R[\frac{1}{p}]}\overline{R}[\frac{1}{p}])$. By Proposition \ref{prop:3.3}, there exists a unique (up to isomorphism) projective \'{e}tale $(\varphi, \mathcal{O}_\mathcal{E})$-module $M$ such that $T^\vee(M) \cong T_p(G)$ as $\mathcal{G}_{\tilde{R}_\infty}$-representations. We remark that if $G$ extends to a $p$-divisible group $G_R$ over $R$, then $T_p(G_R) = T_p(G)$ as $\mathcal{G}_R$-representations.

\section{Extending $p$-divisible Groups} \label{sec:4}

We now prove the generalization of Raynaud's theorem for the relative base $R$ when $e < p-1$, and use an example in \cite{vasiu-zink-purity} on purity of $p$-divisible groups to show that when the ramification is large, such a generalization does not hold. We first consider the special case when the base ring $R_0$ as in Section \ref{sec:2} is equal to the formal power series ring over a Cohen ring.

\begin{prop} \label{prop:4.1}
Suppose $R_0 = \mathcal{O}[\![s_1, \ldots, s_r]\!]$ 	over a Cohen ring $\mathcal{O}$ and $e < p-1$. Let $G$ be a $p$-divisible group over $R[\frac{1}{p}]$, and let $n \geq 1$ be an integer. Suppose that $G[p^n]$ extends to a finite flat group scheme $G_{n, R}$ over $R$. Then for each integer $1 \leq m \leq n$, the group scheme $G_{n, R}[p^m]$ is finite flat over $R$. 

Furthermore, if $H$ is another finite flat group scheme over $R$ extending $G[p^n]$ and if we identify the associated \'etale $(\varphi, \mathcal{O}_{\mathcal{E}})$-modules $M_n \coloneqq \mathfrak{M}^*(G_{R, n})\otimes_{\mathfrak{S}}\mathcal{O}_{\mathcal{E}} = \mathfrak{M}^*(H)\otimes_{\mathfrak{S}}\mathcal{O}_{\mathcal{E}}$, then $\mathfrak{M}^*(G_{R, n}) = \mathfrak{M}^*(H)$ as $\mathfrak{S}$-submodules of $M_n$ with compatible Frobenius.
\end{prop}

\begin{proof}
Let $M$ be the projective \'{e}tale $(\varphi, \mathcal{O}_\mathcal{E})$-module such that $T^\vee(M) = T_p(G)$ as $\mathcal{G}_{\tilde{R}_\infty}$-representations. Denote $\mathfrak{M}_n = \mathfrak{M}^*(G_{n, R})$. Since $T_p(G[p^n]) \cong T_p(G)/p^nT_p(G)$, we have $M_n = \mathfrak{M}_n\otimes_{\mathfrak{S}}\mathcal{O}_\mathcal{E} \cong M/p^nM$ as \'etale $(\varphi, \mathcal{O}_\mathcal{E})$-modules. 

For proving the first statement, we can make the following choice of Frobenius on $R_0$ without loss of generality. Let $k' = \mathcal{O}/(p)$. Note that since $R_0/pR_0 \cong k'[\![s_1, \ldots, s_r]\!]$ has a finite $p$-basis, we have $[k' : k'^p] < \infty$, i.e., $k'$ has a finite $p$-basis. Choose a Frobenius $\varphi_{\mathcal{O}}: \mathcal{O} \rightarrow \mathcal{O}$ lifting the natural Frobenius on $W(k)$, and equip $R_0$ with Frobenius given by $\varphi_{\mathcal{O}}$ and $\varphi(s_i) = s_i^p$. Let $b_0: R_0 \rightarrow \mathcal{O}$ be the $\mathcal{O}$-linear map given by $s_i \mapsto 0$, which is $\varphi$-equivariant. Let $b_g: R_0 \rightarrow R_{0, g}\cong W(k_g)$ be the $\varphi$-equivariant map considered in Section \ref{sec:2}. Note that $\mathfrak{M}_n\otimes_{\mathfrak{S}, b_g}W(k_g)[\![u]\!]$ and $\mathfrak{M}_n\otimes_{\mathfrak{S}, b_0}\mathcal{O}[\![u]\!]$ with the induced diagonal Frobenius are torsion quasi-Kisin modules of height $1$ over $W(k_g)[\![u]\!]$ and $\mathcal{O}[\![u]\!]$ respectively. Denote by $I_j$ the $j$-th Fitting ideal of $\mathfrak{M}_n$ over $\mathfrak{S}_n \coloneqq \mathfrak{S}/p^n\mathfrak{S}$. Let $I_{j, 0}$ and $I_{j, g}$ be the $j$-th Fitting ideal of $\mathfrak{M}_n\otimes_{\mathfrak{S}, b_g}W(k_g)[\![u]\!]$ and $\mathfrak{M}_n\otimes_{\mathfrak{S}, b_0}\mathcal{O}[\![u]\!]$ over $W(k_g)[\![u]\!]/(p^n)$ and $\mathcal{O}[\![u]\!]/(p^n)$ respectively. Then $I_{j, 0}$ and $I_{j, g}$ are given by the images of $I_j$ under the corresponding maps $b_0$ and $b_g$ respectively.

Let $h$ be the height of $G$. Since $e < p-1$, we deduce from \cite[Lemma 4.3.1 and Corollary 4.2.5]{liu-fontaineconjecture} that $\mathfrak{M}_n\otimes_{\mathfrak{S}, b_g}W(k_g)[\![u]\!]$ is free of rank $h$ over $W(k_g)[\![u]\!]/(p^n)$. Furthermore, if we denote by $\mathcal{O}_g$ the $p$-adic completion of $\varinjlim_{\varphi} \mathcal{O}_{(p)}$ with the induced Frobenius and $\kappa \coloneqq \varinjlim_{\varphi} \mathcal{O}/(p)$, then by the universal property of $p$-adic Witt vectors as in Section \ref{sec:2}, $\mathcal{O}_g \cong W(\kappa)$ compatibly with Frobenius endomorphisms. The map $\mathcal{O}[\![u]\!]/(p^n) \rightarrow W(\kappa)[\![u]\!]/(p^n)$ is faithfully flat, and the induced torsion Kisin module $(\mathfrak{M}_n\otimes_{\mathfrak{S}, b_0}\mathcal{O}[\![u]\!])\otimes_{\mathcal{O}[\![u]\!]}W(\kappa)[\![u]\!]$ is free of rank $h$ over $W(\kappa)[\![u]\!]/(p^n)$ by \textit{loc. cit.} Hence, $\mathfrak{M}_n\otimes_{\mathfrak{S}, b_0}\mathcal{O}[\![u]\!]$ is free of rank $h$ over $\mathcal{O}[\![u]\!]/(p^n)$. We obtain
\begin{align*}
	I_{j, g} &= 
	\begin{cases}
	0 & \text{if } j < h\\
	W(k_g)[\![u]\!]/(p^n) & \text{if }	 j \geq h,
	\end{cases}\\
	I_{j, 0} &=
	\begin{cases}
	0 & \text{if } j < h\\
	\mathcal{O}[\![u]\!]/(p^n) & \text{if } j \geq h.
	\end{cases}
\end{align*}  
By Lemma \ref{lem:2.1}, the map $\mathfrak{S}_n \rightarrow W(k_g)[\![u]\!]/(p^n)$ induced from $b_g$ is injective. For $j < h$, the image of $I_j$ under $b_g$ in $W(k_g)[\![u]\!]/(p^n)$ is equal to $I_{j, g}$ which is $0$. Thus, $I_j = 0$ if $j < h$. Suppose $j \geq h$. If $I_j$ is contained in the maximal ideal $(p, s_1, \ldots, s_r, u)$ of $\mathfrak{S}_n$, then the image of $I_j$ under $b_0$ would be contained in the maximal ideal of $\mathcal{O}[\![u]\!]/(p^n)$. Since $I_{j, 0} = \mathcal{O}[\![u]\!]/(p^n)$, we have $I_j = \mathfrak{S}_n$. Hence, $\mathfrak{M}_n$ is projective and thus free of rank $h$ over $\mathfrak{S}_n$. By Theorem \ref{thm:2.6}, $G_{n, R}[p^m]$ is finite flat over $R$ for each $m \geq 1$.

Now we show the second statement, for any choice of Frobenius on $R_0$. Suppose that $G[p^n]$ extends to another finite flat group scheme $H$ over $R$, and let $\mathfrak{N} \coloneqq \mathfrak{M}^*(H)$ be the associated torsion Kisin module. Identify $\mathfrak{N}\otimes_{\mathfrak{S}}\mathcal{O}_{\mathcal{E}} = \mathfrak{M}_n\otimes_{\mathfrak{S}}\mathcal{O}_{\mathcal{E}} = M_n$ as \'etale $(\varphi, \mathcal{O}_{\mathcal{E}})$-modules, and consider both $\mathfrak{N}$ and $\mathfrak{M}_n$ as $\mathfrak{S}_n$-submodules of $M_n$. Since $G_{n, R}[p^m]$ is finite flat over $R$ for each $m \geq 1$ and similarly for $H$, and since $M_n$ is projective over $\mathcal{O}_{\mathcal{E}, n} \coloneqq \mathcal{O}_{\mathcal{E}}/(p^n)$, we have by Theorem \ref{thm:2.6} that $\mathfrak{M}_n$ and $\mathfrak{N}$ are projective and thus flat over $\mathfrak{S}_n$. By \cite[Corollary 4.2.5]{liu-fontaineconjecture}, we have $\mathfrak{M}_n\otimes_{\mathfrak{S}, b_g} W(k_g)[\![u]\!] = \mathfrak{N}\otimes_{\mathfrak{S}, b_g} W(k_g)[\![u]\!]$ as $W(k_g)[\![u]\!]$-submodules of $M_n\otimes_{\mathfrak{S}}W(k_g)[\![u]\!]$. Note that by Lemma \ref{lem:2.1}, the induced map $\mathcal{O}_{\mathcal{E}, n} \rightarrow W_n(k_g)[\![u]\!][\frac{1}{u}]$ is injective, and $\mathcal{O}_{\mathcal{E}, n} \cap W_n(k_g)[\![u]\!] = \mathfrak{S}_n$ as subrings of $W_n(k_g)[\![u]\!][\frac{1}{u}]$. Since $\mathfrak{M}_{n}$ is flat over $\mathfrak{S}_n$, we deduce
\[
(\mathfrak{M}_n\otimes_{\mathfrak{S}_n}\mathcal{O}_{\mathcal{E}, n}) \bigcap (\mathfrak{M}_n\otimes_{\mathfrak{S}_n}W_n(k_g)[\![u]\!]) = \mathfrak{M}_n\otimes_{\mathfrak{S}_n}(\mathcal{O}_{\mathcal{E}, n}\bigcap W_n(k_g)[\![u]\!]) = \mathfrak{M}_n\otimes_{\mathfrak{S}_n}\mathfrak{S}_n = \mathfrak{M}_n
\]	
as $\mathfrak{S}_n$-submodules of $\mathfrak{M}_n\otimes_{\mathfrak{S}_n}W_n(k_g)[\![u]\!][\frac{1}{u}] = M_n\otimes_{\mathfrak{S}}W(k_g)[\![u]\!]$, and similarly
\[
(\mathfrak{N}\otimes_{\mathfrak{S}_n}\mathcal{O}_{\mathcal{E}, n}) \bigcap (\mathfrak{N}\otimes_{\mathfrak{S}_n}W_n(k_g)[\![u]\!]) = \mathfrak{N}
\] 
as $\mathfrak{S}_n$-submodules of $\mathfrak{N}\otimes_{\mathfrak{S}_n}W_n(k_g)[\![u]\!][\frac{1}{u}] = M_n\otimes_{\mathfrak{S}}W(k_g)[\![u]\!]$. Since $\mathfrak{M}_n\otimes_{\mathfrak{S}_n}\mathcal{O}_{\mathcal{E}, n} = M_n = \mathfrak{N}\otimes_{\mathfrak{S}_n}\mathcal{O}_{\mathcal{E}, n}$ and $\mathfrak{M}_n\otimes_{\mathfrak{S}_n}W_n(k_g)[\![u]\!] = \mathfrak{N}\otimes_{\mathfrak{S}_n}W_n(k_g)[\![u]\!]$ as submodules of $M_n\otimes_{\mathfrak{S}}W(k_g)[\![u]\!]$, we obtain $\mathfrak{M}_n = \mathfrak{N}$ with compatible Frobenius. 
\end{proof}

\noindent We remark that in the second statement of above Proposition \ref{prop:4.1}, we do not know whether $\mathfrak{M}^*(G_{R, n}) \cong \mathfrak{M}^*(H)$ as Kisin modules, i.e., whether the connections on both sides are compatible.

Now we consider the general base ring $R$ as in Section \ref{sec:2}.

\begin{thm} \label{thm:4.2}
Assume $e < p-1$. Let $G$ be a $p$-divisible group over $R[\frac{1}{p}]$. Suppose that for each $n$, $G[p^n]$ extends to a finite locally free group scheme $G_{n, R}$ over $R$. Then $G$ extends to a $p$-divisible group over $R$, and such an extension is unique up to isomorphism.

If $e \geq p$ and $R = \mathcal{O}_K[\![s]\!]$, then there exists a $p$-divisible group $G$ over $R[\frac{1}{p}]$ such that $G[p^n]$ extends to a finite locally free group scheme $G_{n, R}$ over $R$ for each $n$ but $G$ does not extend to a $p$-divisible group over $R$. 
\end{thm}

\begin{proof}
Suppose $e < p-1$. Let $M$ be the projective \'{e}tale $(\varphi, \mathcal{O}_\mathcal{E})$-module such that $T^\vee(M) = T_p(G)$ as $\mathcal{G}_{\tilde{R}_\infty}$-representations. For each $n \geq 1$, let $\mathfrak{M}_n \coloneqq \mathfrak{M}^*(G_{n, R}) \in \mathrm{Mod}^{\mathrm{tor}}_{\mathfrak{S}}(\varphi, \nabla)$ be the torsion Kisin module of height $1$ corresponding to $G_{n, R}$. We have $\mathfrak{M}_n\otimes_{\mathfrak{S}}\mathcal{O}_\mathcal{E} \cong M_n \coloneqq M/p^nM$ as \'etale $(\varphi, \mathcal{O}_\mathcal{E})$-modules. Let $h$ be the height of $G$.

For each maximal ideal $\mathfrak{q}$ of $R$, denote $\mathfrak{q}_0 \coloneqq \mathfrak{q} \cap R_0 \subset R_0$ the corresponding maximal ideal of $R_0$, and let $b_{\mathfrak{q}}: R_0 \rightarrow \widehat{R}_{0, \mathfrak{q}_0}$ be the natural $\varphi$-equivariant map where $\widehat{R}_{0, \mathfrak{q}_0}$ denotes the $\mathfrak{q}_0$-adic completion of $R_{0, \mathfrak{q}_0}$. By the structure theorem for complete regular local rings, $\widehat{R}_{0, \mathfrak{q}_0}$ is isomorphic to a formal power series ring  $\widehat{R}_{0, \mathfrak{q}_0} \cong \mathcal{O}[\![s_1, \ldots, s_r]\!]$ over a Cohen ring $\mathcal{O}$. We have the induced base change $b_{\mathfrak{q}}: R \rightarrow \widehat{R}_{\mathfrak{q}} \cong \widehat{R}_{0, \mathfrak{q}_0}\otimes_{W(k)}\mathcal{O}_K$, where $\widehat{R}_{\mathfrak{q}}$ is the $\mathfrak{q}$-adic completion of $R_{\mathfrak{q}}$. Denote $\mathfrak{S}_{\mathfrak{q}} \coloneqq \widehat{R}_{0, \mathfrak{q}_0}[\![u]\!]$. For the $p$-divisible group $G\times_{R[\frac{1}{p}], b_{\mathfrak{q}}}\widehat{R}_{\mathfrak{q}}[\frac{1}{p}]$ over $\widehat{R}_{\mathfrak{q}}[\frac{1}{p}]$, note that $(G\times_{R[\frac{1}{p}]}\widehat{R}_{\mathfrak{q}}[\frac{1}{p}])[p^n]$ extends to the finite locally free group scheme $G_{n, \mathfrak{q}} \coloneqq G_{n, R}\times_{R, b_{\mathfrak{q}}}\widehat{R}_{\mathfrak{q}}$ over $\widehat{R}_{\mathfrak{q}}$ for each $n \geq 1$. By Proposition \ref{prop:4.1}, $G_{n, \mathfrak{q}}[p^m]$ is finite locally free over $\widehat{R}_{\mathfrak{q}}$ for each $m \geq 1$, and thus $\mathfrak{M}^*(G_{n, \mathfrak{q}}) = \mathfrak{M}_n\otimes_{\mathfrak{S}, b_{\mathfrak{q}}}\mathfrak{S}_{\mathfrak{q}}$ is projective over $\mathfrak{S}_{\mathfrak{q}}/(p^n)$ by Theorem \ref{thm:2.6}. Since this holds for each maximal ideal $\mathfrak{q}$ of $R$, we deduce that $\mathfrak{M}_n$ is projective over $\mathfrak{S}/(p^n)$ of rank $h$. In particular, $G_{n, R}[p^m]$ is finite locally free over $R$ for each $m \geq 1$. Note that $G_{n, R}[p^m]\times_R R[\frac{1}{p}] \cong (G_{n, R}\times_R R[\frac{1}{p}])[p^m] \cong G[p^m]$, and $G_{n, R}[p^m]$ has order $p^{mh}$ for each $1 \leq m \leq n$.  

By considering the orders of the groups, we see that the natural sequence of finite locally free group schemes 
\[
0 \rightarrow G_{n+1, R}[p] \rightarrow G_{n+1, R} \rightarrow G_{n+1, R}[p^n] \rightarrow 0,
\]
where the map $G_{n+1, R} \rightarrow G_{n+1, R}[p^n]$ is induced by multiplication by $p$, is short exact. Furthermore, it follows easily from the construction of the functor $\mathfrak{M}^*(\cdot)$ in \cite[Proof of Proposition 9.5]{kim-groupscheme-relative} using isogeny of $p$-divisible groups that $\mathfrak{M}^*(G_{n+1, R}[p]) \cong \mathfrak{M}_{n+1}/p\mathfrak{M}_{n+1}$ as torsion Kisin modules, where $\mathfrak{M}_{n+1}/p\mathfrak{M}_{n+1}$ is equipped with Frobenius and connection induced from $\mathfrak{M}_{n+1}$. Since $\mathfrak{M}^*(\cdot)$ is exact, we have $\mathfrak{M}^*(G_{n+1, R}[p^n]) \cong p\mathfrak{M}_{n+1}$ where $p\mathfrak{M}_{n+1}$ is equipped with Frobenius and connection induced from $\mathfrak{M}_{n+1}$. We claim that $\mathfrak{M}_n \cong p\mathfrak{M}_{n+1}$ as torsion quasi-Kisin modules with compatible Frobenius. Identify $p\mathfrak{M}_{n+1}\otimes_{\mathfrak{S}}\mathcal{O}_\mathcal{E} = M_n = \mathfrak{M}_n\otimes_{\mathfrak{S}}\mathcal{O}_\mathcal{E}$ as \'etale $(\varphi, \mathcal{O}_{\mathcal{E}})$-modules, and consider both $p\mathfrak{M}_{n+1}$ and $\mathfrak{M}_n$ as $\mathfrak{S}_n$-submodules of $M_n$. For the natural injective map $\mathfrak{M}_n \hookrightarrow \mathfrak{M}_n+p\mathfrak{M}_{n+1}$ of $\mathfrak{S}$-modules, consider the induced map $\mathfrak{M}_n\otimes_{\mathfrak{S}, b_{\mathfrak{q}}}\mathfrak{S}_{\mathfrak{q}} \rightarrow (\mathfrak{M}_n+p\mathfrak{M}_{n+1})\otimes_{\mathfrak{S}, b_{\mathfrak{q}}}\mathfrak{S}_{\mathfrak{q}}$ for each maximal ideal $\mathfrak{q}$ of $R$. Since $b_{\mathfrak{q}}: \mathfrak{S} \rightarrow \mathfrak{S}_{\mathfrak{q}}$ is flat, we have $(\mathfrak{M}_n+p\mathfrak{M}_{n+1})\otimes_{\mathfrak{S}, b_{\mathfrak{q}}}\mathfrak{S}_{\mathfrak{q}} = \mathfrak{M}_n\otimes_{\mathfrak{S}, b_{\mathfrak{q}}}\mathfrak{S}_{\mathfrak{q}}+p\mathfrak{M}_{n+1}\otimes_{\mathfrak{S}, b_{\mathfrak{q}}}\mathfrak{S}_{\mathfrak{q}}$, and by Proposition \ref{prop:4.1}, $\mathfrak{M}_n\otimes_{\mathfrak{S}, b_{\mathfrak{q}}}\mathfrak{S}_{\mathfrak{q}} = p\mathfrak{M}_{n+1}\otimes_{\mathfrak{S}, b_{\mathfrak{q}}}\mathfrak{S}_{\mathfrak{q}}$ as submodules of $M_n\otimes_{\mathfrak{S}, b_{\mathfrak{q}}}\mathfrak{S}_{\mathfrak{q}}$. Thus, $\mathfrak{M}_n\otimes_{\mathfrak{S}}\mathfrak{S}_{\mathfrak{q}} \stackrel{\cong}{\rightarrow} (\mathfrak{M}_n+p\mathfrak{M}_{n+1})\otimes_{\mathfrak{S}}\mathfrak{S}_{\mathfrak{q}}$ for each $\mathfrak{q}$, which implies that injective map $\mathfrak{M}_n \hookrightarrow \mathfrak{M}_n+p\mathfrak{M}_{n+1}$ is also surjective. Thus, $p\mathfrak{M}_{n+1} \subset \mathfrak{M}_n$, and similarly $\mathfrak{M}_n \subset p\mathfrak{M}_{n+1}$. This shows the claim $\mathfrak{M}_n = p\mathfrak{M}_{n+1}$ with compatible Frobenius. 

Thus, $\mathfrak{M} \coloneqq \varprojlim_{n} \mathfrak{M}_n$ with the induced Frobenius is a quasi-Kisin module of height $1$ over $\mathfrak{S}$. We now equip $\mathcal{M} \coloneqq \mathfrak{M}\otimes_{\mathfrak{S}, \varphi}R_0$ with a connection. Denote by $\nabla_{\mathfrak{M}_n}: \mathfrak{M}_n\otimes_{\mathfrak{S}, \varphi}R_0 \rightarrow (\mathfrak{M}_n\otimes_{\mathfrak{S}, \varphi}R_0)\otimes_{R_0}\widehat{\Omega}_{R_0}$ the connection for the torsion Kisin module $\mathfrak{M}_n$, and let $\mathcal{M}_n = \mathcal{M}\otimes_{R_0}R_0/(p^n)$. Consider the multiset 
\[
S_n = \{\nabla_{\mathfrak{M}_k}\otimes_{R_0}R_0/(p^n) ~|~ k \geq n+1 \}
\]
of connections on $\mathcal{M}_n$. Note that for each $k \geq n+1$, the connection $\nabla_{\mathfrak{M}_k}\otimes_{R_0}R_0/(p^n)$ satisfies the commutative diagram (\ref{eq:2.1}) in Section \ref{sec:2}. Using the result discussed at the end of Section \ref{sec:2}, we choose a compatible system of connections $\nabla_n$ on $\mathcal{M}_n$ inductively as follows. Identify $\widehat{\Omega}_{R_0} = \bigoplus_{i=1}^d R_0 \cdot d(\log{t_i})$. Let $\mathcal{M}^1 \subset \mathcal{M}$ be a direct factor lifting $\mathrm{Fil}^1 \mathcal{M}/p\mathcal{M} \subset \mathcal{M}/p\mathcal{M}$ as in Section \ref{sec:2}, and we fix a choice of a finite Zariski covering of $\mathrm{Spf}(R_0, p)$ over which $\mathcal{M}^1$ and $\mathcal{M}/\mathcal{M}^1$ are free, and fix a basis of $\mathcal{M}$ adapted to $\mathcal{M}^1$ after passing to the covering. For $n = 1$, $S_1$ is finite as a set of connections on $\mathcal{M}_1$, and we choose a connection $\nabla_1$ on $\mathcal{M}_1$ which has infinite multiplicity in the multiset $S_1$. When we are given a choice of connection $\nabla_n$ on $\mathcal{M}_n$, the elements in $S_{n+1}$ which lift $\nabla_n$ are contained in a finite set of connections, and we choose a connection $\nabla_{n+1}$ on $\mathcal{M}_{n+1}$ which has infinite multiplicity in $S_{n+1}$. Let $\nabla \coloneqq \varprojlim_{n} \nabla_n$ be the induced connection on $\mathcal{M}$. Then $\nabla$ is compatible with Frobenius, integrable, and topologically quasi-nilpotent. Hence, $(\mathfrak{M}, \nabla)$ is a Kisin module of height $1$, and the corresponding $p$-divisible group over $R$ extends $G$. The uniqueness of extending $G$ up to isomorphism follows from \cite[Theorem 4]{tate}.

On the other hand, assume $e \geq p$ and $R_0 = W(k)[\![s]\!]$. Let $U = \mathrm{Spec}R \backslash \{\mathfrak{m}\}$ be the open subscheme of $\mathrm{Spec}R$, where $\mathfrak{m}$ is the closed point given by the maximal ideal of $R$. By \cite[Theorem 28]{vasiu-zink-purity}, there exists a $p$-divisible group $G_U$ over $U$ which does not extend to a $p$-divisible group over $R$. By \cite[Chapter V. Lemma 6.2]{faltings-chai}, for each $n \geq 1$, the finite locally free group scheme $G_U[p^n]$ extends uniquely to a finite locally free group scheme over $R$ (if $A$ denotes the Hopf algebra for $G_U[p^n]\times_U R[\frac{1}{p}]$ and $B$ denotes the Hopf algebra for $G_U[p^n]\times_U R[\frac{1}{s}]$, then identifying $C \coloneqq A[\frac{1}{s}] = B[\frac{1}{p}]$ as the Hopf algebra for $G_U[p^n]\times_U R[\frac{1}{p}][\frac{1}{s}]$, the unique extension is given by $A \cap B$ with the induced Hopf algebra structure over $R$). Let $G = G_U \times_U R[\frac{1}{p}]$ be the $p$-divisible group over $R[\frac{1}{p}]$, and suppose $G$ extends to a $p$-divisible group $G_R$ over $R$. Since $G_U \times_U (R[\frac{1}{s}])[\frac{1}{p}] = G_R \times_R (R[\frac{1}{s}])[\frac{1}{p}]$, we have by \cite[Theorem 4]{tate} that $G_U \times_U R[\frac{1}{s}] = G_R \times_R R[\frac{1}{s}]$. Thus, $G_R\times_R U = G_U$, which contradicts to that $G_U$ does not extend over $R$. This shows that $G$ cannot be extended to a $p$-divisible group over $R$. 
\end{proof}

\section{Barsotti-Tate Deformation Ring for Relative Base of Dimension 2} \label{sec:5}

Throughout this section, we assume that the Krull dimension of $R$ is equal to $2$. For a finite $\mathbf{Q}_p$-representation $V$ of $\mathcal{G}_R$, we say it is \textit{Barsotti-Tate} if there exists a $p$-divisible group $G_R$ over $R$ such that $V = T_p(G_R)\otimes_{\mathbf{Z}_p}\mathbf{Q}_p$ as $\mathcal{G}_R$-representations.

\begin{prop} \label{prop:5.1}
Assume $e< p-1$. Let $T$ be a finite free $\mathbf{Z}_p$-representation of $\mathcal{G}_R$ such that $T[\frac{1}{p}]$ is Barsotti-Tate. Then there exists a $p$-divisible group $G_R$ over $R$ such that $T = T_p(G_R)$.	
\end{prop}

\begin{proof}
Since $T[\frac{1}{p}]$ is Barsotti-Tate, there exists a $p$-divisible group $G'_R$ over $R$ such that $T_p(G'_R)[\frac{1}{p}] = T[\frac{1}{p}]$. Denote $T' = T_p(G'_R), ~G' = G'_R\times_R R[\frac{1}{p}]$, and let $G$ be the $p$-divisible group over $R[\frac{1}{p}]$ corresponding to the representation $T$. 

Since $p^nT \subset T'$ and $p^nT' \subset T$ for some positiver integer $n$, we have an isogeny $f: G' \rightarrow G$. Let $H \coloneqq \mathrm{ker}(f)$, which is a finite locally free group scheme over $R[\frac{1}{p}]$. Then we have a closed immersion $h: H \hookrightarrow G'[p^m]$ for some positive integer $m$. Note that $G'[p^m]$ extends to the finite locally free group scheme $G'_R[p^m]$ over $R$.

Let $H_R$ be the scheme theoretic closure of $H$ over $R$ obtained from $h$ and $G'_R[p^m]$, given similarly as in \cite[Section 2.1]{raynaud-groupscheme}. By the construction of the scheme theoretic closure, $H_R$ is a finite group scheme. We claim that it is locally free over $R$. For that, let $\mathfrak{q}$ be a maximal ideal of $R$ and let $\mathfrak{q}_0 = \mathfrak{q} \cap R_0$, and consider the base change map $b_{\mathfrak{q}}: R \rightarrow \widehat{R}_{\mathfrak{q}}$ as in the proof of Theorem \ref{thm:4.2}. Since $R$ has Krull dimension $2$, we have $\widehat{R}_{\mathfrak{q}} \cong \mathcal{O}_{\mathfrak{q_0}}[\![s]\!]\otimes_{W(k)}\mathcal{O}_K$ for some Cohen ring $\mathcal{O}_{\mathfrak{q_0}}$ with the maximal ideal $(p)$. Let $U_{\mathfrak{q}} \subset \mathrm{Spec}\widehat{R}_{\mathfrak{q}}$ be the closed subscheme obtained by deleting the closed point given by $\mathfrak{q}$. Since $U_{\mathfrak{q}}$ is a Dedekind scheme, $(H_R \times_{R} \widehat{R}_{\mathfrak{q}})\otimes_{\widehat{R}_{\mathfrak{q}}} U_{\mathfrak{q}}$ is locally free over $U_{\mathfrak{q}}$ as the corresponding sheaf of Hopf algebras is torsion free. It extends uniquely to a finite locally free group scheme $H_{\mathfrak{q}}$ over $\widehat{R}_{\mathfrak{q}}$ by \cite[Chapter V. Lemma 6.2]{faltings-chai}. On the other hand, since $e < p-1$, note that $p \notin (\mathfrak{q}\widehat{R}_{\mathfrak{q}})^{p-1}$. Since $h$ is a monomorphism, we deduce from \cite[Proposition 15]{vasiu-zink-purity} applied for $\widehat{R}_{\mathfrak{q}}$ that the map $H_{\mathfrak{q}} \rightarrow G'_R[p^m]\times_R \widehat{R}_{\mathfrak{q}}$ of finite flat group schemes is a monomorphism and hence a closed immersion. Thus, $H_R \times_R \widehat{R}_{\mathfrak{q}} = H_{\mathfrak{q}}$. Since this holds for every maximal ideal $\mathfrak{q}$ of $R$, $H_R$ is locally free over $R$. 

The map $h$ induces a closed immersion $H_R \hookrightarrow G'_R[p^m]$, and $G_R \coloneqq G'_R/H_R$ is a $p$-divisible group over $R$. It is clear from the construction that $T_p(G_R) = T$ as $\mathbf{Z}_p[\mathcal{G}_R]$-modules.  
\end{proof}

For a finite free $\mathbf{Z}_p$-representation $T$ of $\mathcal{G}_R$, it makes sense by Proposition \ref{prop:5.1} to say that $T$ is \textit{Barsotti-Tate} if there exists a $p$-divisible group $G_R$ over $R$ such that $T = T_p(G_R)$.  

\begin{lem} \label{lem:5.2}
Assume $e < p-1$. Let $H_R$ be a $p$-power order finite locally free group scheme over $R$, and let $T = H_R(\overline{R})$ be the corresponding torsion $\mathbf{Z}_p$-representation of $\mathcal{G}_R$. If we have a short exact sequence of $\mathbf{Z}_p[\mathcal{G}_R]$-modules
\[
0 \rightarrow T_1 \rightarrow T \rightarrow T_2 \rightarrow 0,
\]	
then there exist $p$-power order finite locally free group schemes $H_{1, R}$ and $H_{2, R}$ over $R$ such that $T_i = H_{i, R}(\overline{R})$ for $i = 1, 2$ as $\mathcal{G}_R$-representations.
\end{lem}

\begin{proof}
Let $H \coloneqq H_R\times_R R[\frac{1}{p}]$. Let $H_i$ for $i = 1, 2$ be finite locally free group schemes over $R[\frac{1}{p}]$ such that $H_i(\overline{R}[\frac{1}{p}]) = T_i$ as $\mathcal{G}_R$-representations. The given exact sequence of $\mathcal{G}_R$-representations induce the short exact sequence
\[
0 \rightarrow H_1 \rightarrow H \rightarrow H_2 \rightarrow 0
\]
of finite locally free group schemes. Let $H_{1, R}$ be the scheme theoretic closure of $H_1$ over $R$ obtained from the closed embedding $H_1 \hookrightarrow H$ and $H_R$. By the same argument as in the proof of Proposition \ref{prop:5.1}, $H_{1, R}$ is a finite locally free group scheme over $R$ extending $H_1$. Furthermore, $H_{2, R} \coloneqq H_R/H_{1, R}$ is a finite locally free group scheme over $R$ extending $H_2$ (cf. \cite{raynaud-quotgp}). It is clear that $T_i = H_{i, R}(\overline{R})$ for $i = 1, 2$.  
\end{proof}

\begin{cor} \label{cor:5.3}
Assume $e< p-1$. Let $A_1 \hookrightarrow A_2$ be an injective map of finite free $\mathbf{Z}_p$-algebras. Let $T_{A_1}$ be a finite free $A_1$-module given the $p$-adic topology and equipped with a continuous $A_1$-linear $\mathcal{G}_R$-action. Let $T_{A_2} \coloneqq T_{A_1} \otimes_{A_1} A_2$ be the induced representation with the $A_2$-linear $\mathcal{G}_R$-action. Then $T_{A_1}$ is Barsotti-Tate if and only if $T_{A_2}$ is Barsotti-Tate.
\end{cor}

\begin{proof}
Let $G_2$ be the $p$-divisible group over $R[\frac{1}{p}]$ corresponding to $T_{A_2}$. Suppose first that $T_{A_1}$ is Barsotti-Tate. Note that there exist finitely many elements $x_1, \ldots, x_m \in A_2$ generating $A_2$ as an $A_1$-module. We have a surjective map of $\mathbf{Z}_p[\mathcal{G}_R]$-modules $T_{A_1}^m \twoheadrightarrow T_{A_2}$ sending the canonical basis elements $e_i$ of $T_{A_1}^m$ to $x_i$. Note that the direct sum representation $T_{A_1}^m$ is Barsotti-Tate. For each integer $n \geq 1$, $T_{A_2}/p^n$ is therefore a quotient of $T_{A_1}^m/p^n$, and by Lemma \ref{lem:5.2}, $G_2[p^n]$ extends to a finite locally free group scheme over $R$. Thus, $T_{A_2}$ is Barsotti-Tate by Theorem \ref{thm:4.2}.

Conversely, suppose $T_{A_2}$ is Barsotti-Tate. Let $B_3$ be the quotient of the induced injection $A_1[\frac{1}{p}] \hookrightarrow A_2[\frac{1}{p}]$ of $\mathbf{Q}_p$-algebras, and let $T \subset T_{A_2}$ be the kernel of the induced map of representations $T_{A_2} \rightarrow T_{A_2}\otimes_{A_2} B_3$. Then for each integer $n \geq 1$, the map $T/p^n \rightarrow T_{A_2}/p^n$ is injective. Hence, by Lemma \ref{lem:5.2} and Theorem \ref{thm:4.2} similarly as above, $T$ is Barsotti-Tate. Since $T[\frac{1}{p}] = T_{A_1}[\frac{1}{p}]$, $T_{A_1}$ is Barsotti-Tate by Proposition \ref{prop:5.1}. 
\end{proof}

We now study the geometry of the locus of Barsotti-Tate representations. Denote by $\mathcal{C}$ the category of topological local $\mathbf{Z}_p$-algebras $A$ satisfying the following conditions:
\begin{itemize}
\item the natural map $\mathbf{Z}_p \rightarrow A/\mathfrak{m}_A$ is surjective, where $\mathfrak{m}_A$ denotes the maximal ideal of $A$;
\item the map from $A$ to the projective limit of its discrete artinian quotients is a topological isomorphism.	
\end{itemize}

\noindent Note that the first condition implies that the residue field of $A$ is $\mathbf{F}_p$. The second condition is equivalent to the condition that $A$ is complete and its topology can be given by a collection of open ideals $\mathfrak{a} \subset A$ for which $A/\mathfrak{a}$ is aritinian. Morphisms in $\mathcal{C}$ are continuous $\mathbf{Z}_p$-algebra morphisms. The following proposition is shown in \cite{smit}.

\begin{prop} \label{prop:5.4} \emph{(cf. \cite[Proposition 2.4]{smit})}
Suppose $A$ is a Noetherian ring in $\mathcal{C}$. Then the topology on $A$ is equal to the $\mathfrak{m}_A$-adic topology. 
\end{prop}

For $A \in \mathcal{C}$, we mean by an $A$-\textit{representation of} $\mathcal{G}_R$ a finite free $A$-module equipped with a continuous $A$-linear $\mathcal{G}_R$-action. We fix an $\mathbf{F}_p$-representation $V_0$ of $\mathcal{G}_R$ which is absolutely irreducible. For $A \in \mathcal{C}$, a \textit{deformation} of $V_0$ in $A$ is an isomorphism class of $A$-representations of $V$ of $\mathcal{G}_R$ satisfying $V\otimes_A \mathbf{F}_p \cong V_0$ as $\mathbf{F}_p[\mathcal{G}_R]$-modules. We denote by $\mathrm{Def}(V_0, A)$ the set of such deformations. A morphism $f: A \rightarrow A'$ in $\mathcal{C}$ induces a map $f_*: \mathrm{Def}(V_0, A) \rightarrow \mathrm{Def}(V_0, A')$ sending the class of an $A$-representation $V$ to the class of $V\otimes_{A, f}A'$. The following theorem on universal deformation ring is proved in \cite{smit}.

\begin{thm} \label{thm:5.5} \emph{(cf. \cite[Theorem 2.3]{smit})} 
There exists a universal deformation ring $A_{\mathrm{univ}} \in \mathcal{C}$ and a deformation $V_{\mathrm{univ}} \in \mathrm{Def}(V_0, A_{\mathrm{univ}})$ such that for all $A \in \mathcal{C}$, we have a bijection
\begin{equation} \label{eq:5.1}
\mathrm{Hom}_{\mathcal{C}}(A_{\mathrm{univ}}, A) \stackrel{\cong}{\rightarrow} \mathrm{Def}(V_0, A)
\end{equation}
given by $f \mapsto f_*(V_{\mathrm{univ}})$. 
\end{thm}

\noindent We remark that $A_{\mathrm{univ}}$ is Noetherian if and only if $\mathrm{dim}_{\mathbf{F}_p}H^1(\mathcal{G}_R, \mathrm{End}_{\mathbf{F}_p}(V_0))$ is finite (cf. \textit{loc. cit.}). Thus, $A_{\mathrm{univ}}$ is not Noetherian in general, even when $R = \mathcal{O}_K$ if $K/\mathbf{Q}_p$ is infinite.
   
Let $\mathcal{C}^0$ be the full subcategory of $\mathcal{C}$ consisting of artinian rings. Abusing the notation, we write $V \in \mathrm{Def}(V_0, A)$ for an $A$-representation $V$ to mean that $V\otimes_A \mathbf{F}_p \cong V_0$. For $A \in \mathcal{C}^0$ and a representation $V_A \in \mathrm{Def}(V_0, A)$, we say $V_A$ is \textit{torsion Barsotti-Tate} if there exists a $p$-power order finite locally free group scheme $H_R$ over $R$ such that $V_A \cong H_R(\overline{R})$ as $\mathbf{Z}_p[\mathcal{G}_R]$-modules. We remark that if $R$ is local, then every $p$-power order finite locally free group scheme over $R$ embeds into a $p$-divisible group over $R$, and thus $V_A$ is torsion Barsotti-Tate if and only if it is a quotient of a finite free $\mathbf{Z}_p$-representation which is Barsotti-Tate. For $A \in \mathcal{C}$, denote by $\mathrm{BT}(V_0, A)$ the subset of $\mathrm{Def}(V_0, A)$ consisting of the isomorphism classes of representations $V_A$ such that $V_A\otimes_A A/\mathfrak{a}$ is torsion Barsotti-Tate for all open ideals $\mathfrak{a} \subsetneq A$.

\begin{prop} \label{prop:5.6}
Assume $e < p-1$. For any $\mathcal{C}$-morphism $f: A \rightarrow A'$, we have $f_*(\mathrm{BT}(V_0, A)) \subset \mathrm{BT}(V_0, A')$. Furthermore, there exists a closed ideal $\mathfrak{a}_{\mathrm{BT}}$ of the universal deformation ring $A_{\mathrm{univ}}$ such that the map (\ref{eq:5.1}) induces a bijection $\mathrm{Hom}_{\mathcal{C}}(A_{\mathrm{univ}}/\mathfrak{a}_{\mathrm{BT}}, A) \stackrel{\cong}{\rightarrow} \mathrm{BT}(V_0, A)$. 	
\end{prop}
 
\begin{proof}
We check the conditions in \cite[Section 6]{smit}. Let $f: A \hookrightarrow A'$ be an injective morphism of artinian rings in $\mathcal{C}$, and let $V_A \in \mathrm{Def}(V_0, A)$ be a representation. We first claim that $V_A \in \mathrm{BT}(V_0, A)$ if and only if $V_{A'} \coloneqq V_A\otimes_{A, f} A' \in \mathrm{BT}(V_0, A')$. Suppose that $V_A \in \mathrm{BT}(V_0, A)$. Note that $A'$ is a finite $A$-module. Let $x_1, \ldots, x_m$ generate $A'$ over $A$. Then we have a surjective map of $\mathbf{Z}_p[\mathcal{G}_R]$-modules $V_A^m \twoheadrightarrow V_{A'}$ sending the canonical basis elements $e_i$ of $V_A^m$ for $i = 1, \ldots, m$ to $x_i$. Since $V_A^m$ is the direct sum of $m$-copies of $V_A$, it is torsion Barsotti-Tate. Thus, by Lemma \ref{lem:5.2}, $V_{A'} \in \mathrm{BT}(V_0, A')$. Conversely, suppose $V_{A'}\in \mathrm{BT}(V_0, A')$. Since we have an injective map of $\mathbf{Z}_p[\mathcal{G}_R]$-modules $V_A \hookrightarrow V_{A'}$, we get $V_A \in \mathrm{BT}(V_0, A)$ by Lemma \ref{lem:5.2}.

Now, for $A \in \mathcal{C}$ and a representation $V_A \in \mathrm{Def}(V_0, A)$, suppose $\mathfrak{a}_1, \mathfrak{a}_2 \subsetneq A$ are open ideals such that $V_A\otimes_A (A/\mathfrak{a}_i) \in \mathrm{BT}(V_0, A/\mathfrak{a}_i)$ for $i = 1, 2$. The natural map $A/(\mathfrak{a}_1\cap \mathfrak{a}_2) \rightarrow A/\mathfrak{a}_1 \oplus A/\mathfrak{a}_2$ is injective, and it induces the injective map of $\mathbf{Z}_p[\mathcal{G}_R]$-modules
\[
V_A\otimes_A A/(\mathfrak{a}_1\cap \mathfrak{a}_2) \hookrightarrow (V_A\otimes_A A/\mathfrak{a}_1) \oplus (V_A\otimes_A A/\mathfrak{a}_2).
\] 	
Since the direct sum $(V_A\otimes_A A/\mathfrak{a}_1) \oplus (V_A\otimes_A A/\mathfrak{a}_2)$ is torsion Barsotti-Tate, we see from Lemma \ref{lem:5.2} that $V_A\otimes_A A/(\mathfrak{a}_1\cap \mathfrak{a}_2) \in \mathrm{BT}(V_0, A/(\mathfrak{a}_1\cap \mathfrak{a}_2))$.

The assertion then follows from \cite[Proposition 6.1]{smit}.
\end{proof}

We now show that when $e < p-1$, the locus of Barsotti-Tate representations cuts out a closed subscheme of the universal deformation scheme $\mathrm{Spec} (A_{\mathrm{univ}})$:

\begin{thm}
Suppose $e < p-1$ (and recall that the Krull dimension of $R$ is assumed to be equal to $2$). Let $A$ be a finite flat $\mathbf{Z}_p$-algebra equipped with the $p$-adic topology, and let $f: A_{\mathrm{univ}} \rightarrow A$ be a continuous $\mathbf{Z}_p$-algebra homomorphism. Then the induced representation $V_{\mathrm{univ}}\otimes_{A_{\mathrm{univ}}, f} A[\frac{1}{p}]$ of $\mathcal{G}_R$ is Barsotti-Tate if and only if $f$ factors through the quotient $A_{\mathrm{univ}}/\mathfrak{a}_{\mathrm{BT}}$.
\end{thm}

\begin{proof}
Let $A_1 \coloneqq \mathrm{im}(f) \subset A$, and let $T_{A_1} = V_{\mathrm{univ}}\otimes_{A_{\mathrm{univ}}, f} A_1$. Then $T_{A_1}\otimes_{A_1} A = V_{\mathrm{univ}}\otimes_{A_{\mathrm{univ}}, f} A$, and by Proposition \ref{prop:5.1} and Corollary \ref{cor:5.3}, it suffices to show that $T_{A_1}$ is Barsotti-Tate if and only if $f$ factors through $A_{\mathrm{univ}}/\mathfrak{a}_{\mathrm{BT}}$. Note that $A_1 \in \mathcal{C}$, and since $A_1$ is finite flat over $\mathbf{Z}_p$, the topology on $A_1$ is equivalent to the $p$-adic topology and $f: A_{\mathrm{univ}} \rightarrow A_1$ is continuous by Proposition \ref{prop:5.4}. Suppose first that $T_{A_1}$ is Barsotti-Tate, so that there exists a $p$-divisible group $G_R$ over $R$ such that $T_p(G_R) \cong T_{A_1}$. For each integer $n \geq 1$, we then have $(V_{\mathrm{univ}}\otimes_{A_{\mathrm{univ}}, f} A_1)\otimes_{A_1} A_1/(p^n) = T_{A_1}/p^n \cong (G_R[p^n])(\overline{R})$, so $V_{\mathrm{univ}}\otimes_{A_{\mathrm{univ}}, f} A_1/(p^n) \in \mathrm{BT}(V_0, A_1/(p^n))$. Hence, by Proposition \ref{prop:5.6}, $f$ factors through $A_{\mathrm{univ}}/\mathfrak{a}_{\mathrm{BT}}$.

Conversely, suppose $f$ factors through $A_{\mathrm{univ}}/\mathfrak{a}_{\mathrm{BT}}$. Let $G$ be the $p$-divisible group over $R[\frac{1}{p}]$ corresponding to $T_{A_1}$. For each $n \geq 1$, $T_{A_1}/p^n$ is torsion Barsotti-Tate by Proposition \ref{prop:5.6}, so $G[p^n]$ extends to a finite locally free group scheme over $R$. Then by Theorem \ref{thm:4.2}, $T_{A_1}$ is Barsotti-Tate.  
\end{proof}

On the other hand, if the ramification is large, we can deduce that the locus of Barsotti-Tate representations is not $p$-adically closed in general:

\begin{prop} \label{prop:5.8}
Let $R = \mathcal{O}_K[\![s]\!]$ and suppose $e \geq p$. There exists a $\mathbf{Z}_p$-representation $T$ of $\mathcal{G}_R$ such that $T/p^nT$ is torsion Barsotti-Tate for each $n \geq 1$ but $T$ is not Barsotti-Tate.	
\end{prop}

\begin{proof}
By Theorem \ref{thm:4.2}, there exists a $p$-divisible group $G$ over $R[\frac{1}{p}]$ such that $G[p^n]$ extends to a finite locally free group scheme $G_{n, R}$ over $R$ but $G$ does not extend to a $p$-divisible group over $R$. Let $T$ be the representation corresponding to $G$. Then for each $n$, we have $T/p^nT \cong G_{n, R}(\overline{R})$ so it is torsion Barsotti-Tate. However, $T$ is not Barsotti-Tate since $G$ does not extend over $R$.	
\end{proof}

\bibliographystyle{amsalpha}
\bibliography{library}
	
\end{document}